\documentclass[11pt,a4paper,oneside, colophon]{article}
\usepackage[utf8]{inputenc}
\usepackage[english]{babel}
\usepackage{amsfonts}
\usepackage{amsmath}
\usepackage{amssymb}
\usepackage{amsthm}
\usepackage{comment}
\usepackage{float}
\usepackage{textcomp}
\usepackage{faktor}
\usepackage{setspace}
\usepackage[dvipsnames]{xcolor}
\usepackage{graphicx}
\usepackage{fancybox}
\usepackage{mathrsfs}
\usepackage{mathtools}
\usepackage{tikz}
\usepackage{centernot}
\usepackage{esint}
\usepackage{algpseudocode}
\usepackage{framed}
\usepackage{hyperref}
\usepackage[bottom]{footmisc}
\usepackage[square, sort, numbers]{natbib}
\usepackage[T1]{fontenc}
\usepackage{mathtools}
\usepackage{amssymb}
\usepackage{graphicx}
\usepackage{epstopdf}
\usepackage{listings}
\usepackage{fancyhdr}
\usepackage{tikz}
\usepackage{afterpage}
\usepackage{ dsfont }
\usepackage[a4paper,top=2.5cm,bottom=2.5cm,left=2.5cm,right=2.5cm]{geometry}

\tikzstyle{level 1}=[level distance=4cm, sibling distance=3.5cm,->]
\tikzstyle{level 2}=[level distance=4cm, sibling distance=2cm,->]

\tikzstyle{bag} = [text width=2em, text centered]
\tikzstyle{end} = []

\DeclarePairedDelimiter{\set}{ \{ } { \} }

\newcommand{\sfrac}[2]{ 
	{\raise0.8ex\hbox{$#1$} \!\mathord{\left/ 
			{\vphantom {#1 #2}}\right.\kern-\nulldelimiterspace} 
		\!\lower0.8ex\hbox{$#2$}} 
}

\newtheorem{teo}{Theorem}[section]
\newtheorem{lema}[teo]{Lemma}

\newtheorem{prop}[teo]{Proposition}

\theoremstyle{definition}
\newtheorem{defin}[teo]{Definition}
\newtheorem{notz}{Notation}
\newtheorem{obs}[teo]{Remark}

\newcommand{\N}{\mathbb{N} }

\newcommand{\R}{\mathbb{R} }
\newcommand{\E}{\mathbb{E} }
\newcommand{\bbG}{\mathbb{G} }
\newcommand{\bbF}{\mathbb{F} }
\newcommand{\bbX}{\mathbb{X}}
\newcommand{\F}{\mathcal{F}}
\newcommand{\calG}{\mathcal{G}}
\newcommand{\calA}{\mathcal{A}}
\newcommand{\calD}{\mathscr{D}}
\newcommand{\calU}{\mathcal{U}}
\newcommand{\calL}{\mathcal{L}}
\newcommand{\calH}{\mathcal{H}}

\newcommand{\hatu}{\hat{u}}
\newcommand{\hatX}{\hat{X}}

\newcommand{\hatkappa}{\hat{\kappa}}

\newcommand{\hatk}{\hat{\kappa}}
\newcommand{\hatb}{\hat{b}}

\newcommand{\tH}{\widetilde{H}}

\newcommand{\esssup}{\text{ess\,sup}}

\DeclareRobustCommand{\rchi}{{\mathpalette\irchi\relax}}
\newcommand{\irchi}[2]{\raisebox{\depth}{$#1\chi$}} 

\DeclareMathSymbol{\shortminus}{\mathbin}{AMSa}{"39}

\newcommand{\probspace}{(\Omega,\F,P)}
\renewcommand{\t}{t \in [0,T]}

\usepackage{pict2e,picture}

\makeatletter
\DeclareRobustCommand{\Arrow}[1][]{%
	\check@mathfonts
	\if\relax\detokenize{#1}\relax
	\settowidth{\dimen@}{$\m@th\rightarrow$}%
	\else
	\setlength{\dimen@}{#1}%
	\fi
	\sbox\z@{\usefont{U}{lasy}{m}{n}\symbol{41}}%
	\begin{picture}(\dimen@,\ht\z@)
		\roundcap
		\put(\dimexpr\dimen@-.7\wd\z@,0){\usebox\z@}
		\put(0,\fontdimen22\textfont2){\line(1,0){\dimen@}}
	\end{picture}%
}
\title{Stochastic Volterra equations with time-changed Lévy noise and maximum principles}
\author{Giulia di Nunno\thanks{Department of Mathematics, University of Oslo, P:O: Box 1053 Blindern, N-0316 Oslo, Email: giulian@math.uio.no.} \thanks{Department of Business and Management Science, NHH Norwegian School of Economics, Helleveien 30, N-5045 Bergen.}\and Michele Giordano\thanks{Department of Mathematics, University of Oslo, P:O: Box 1053 Blindern, N-0316 Oslo, Email: michelgi@math.uio.no}}
\date{February 8th 2023}

\begin{document}
	\allowdisplaybreaks
	\numberwithin{equation}{section}
	\thispagestyle{empty}
	\vspace{1 cm}
	\maketitle
	\begin{abstract}
	\noindent Motivated by a problem of optimal harvesting of natural resources, we study a control problem for Volterra type dynamics driven by time-changed Lévy noises, which are in general not Markovian. To exploit the nature of the noise, we make use of different kind of information flows within a maximum principle approach. For this 
	we work with backward stochastic differential equations (BSDE) with time-change and exploit the \emph{non-anticipating} stochastic derivative introduced in \cite{DiNunno2}. We prove both a sufficient and necessary stochastic maximum principle.
		\end{abstract}

	\section{Introduction}
	Optimal harvesting is a fairly classical problem in control theory and it is still a timely question to address when thinking of sustainability in the management of natural resources. 
 In this work we deal with a problem of optimal harvesting from a population, the growth of which is modelled by Volterra time dynamics of the type
\begin{equation}\label{X(t) example}
	    X(t) = X_0+\int_0^t \left(r(t,s) - K u(s)\right)X(s) ds +\int_0^t\sigma(s)X(s)dB(s), \quad \t.
	\end{equation}
The term $r$ represents the growth rate, the constant $K$ is the catchability coefficient, and the control $u$ is the fishing effort.
The Volterra structure is inherited from the deterministic analogous models that can be found, e.g., in \cite{VolterraFish,VolterraFish2,VolterrraIntegral}. As we can see, this form of time dependence is often used in the description of fish populations. When considering fish as a commodity, the modelling of fish population is representing the possible dynamics of offer, in the interplay between offer and demand.
In our work, however, we consider Volterra \emph{stochastic} integral equations, which represent a natural extension including the uncertainty of the environment influencing the population growth. For this we are motivated by \cite{StochasticFish,StochasticFish2}.

Our model has an element of novelty with respect to the others presented. This is given by the nature of the noise $B$ which is associated to a time-changed Brownian motion. This is well motivated by the clustering effects that such noises can described. 
For the description on how time-change helps to described clustering, we can refer to a first discussion in \cite[Chapter IV, 3e]{Shiryaev} and a more recent study \cite[Chapter 3]{Saakvitne} in the context of market microstructure. Within population dynamics the evidence of clustering is largely discussed in the recent  literature in biology and ecology. See just as example \cite{logistic}. 

We remark that in the literature of mathematical finance, dynamics of the form \eqref{X(t) example}, but with L\'evy type noises were used in models \cite{Oksendal_Forward}. On the other side, time-change has been suggested in the study of volatility modelling, e.g. \cite{BNS, CW,GMY,Swish, veraart-tc}, energy markets, e.g. \cite{electricity}, and default models, e.g.\cite{default}. Also it is used in kinetic theory, see e.g. \cite{menonquick}.

Keeping our motivation in mind, we treat here stochastic control for general Volterra type dynamics, allowing also for jumps:
	\begin{equation}\label{X_t introduction}
	X^u(t)=X_0+\int_0^t b(t,s,\lambda_s,u(s),X^u(s\shortminus))ds+\int_0^t\!\!\!\int_{\mathbb{R}} \kappa(t,s,z,\lambda_s,u(s),X^u(s\shortminus))\mu(dsdz),
	\end{equation}	
	where the driving noise $\mu$ is given by the random measure
	\begin{equation}\label{mu introduction}
		\mu(\Delta)=B(\Delta \cap [0,T]\times \{0\})+\widetilde{H}(\Delta \cap [0,T]\times \mathbb{R}_0), \quad \Delta\in\mathcal{B}([0,T]\times \mathbb{R}),
	\end{equation} 
	which is the mixture of a \emph{conditional Gaussian measure} $B$ on $[0,T]\times\{0\}$ and a \emph{conditional centered Poisson measure} $\widetilde{H}$ on $[0,T]\times \mathbb{R}_0$. Here $\R_0:=\mathbb{R}\backslash\{0\}$ and $\mathcal {B}$ represents the Borel $\sigma$-algebra. Both $B$ and $\widetilde{H}$ are set in relationship with a time-changed Brownian motion and time-changed Poisson measure, respectively, via Theorem 3.1 in \cite{Serfozo} (see also \cite{Grigelionis}). Note that the coefficients in \eqref{X_t introduction} may also depend on the time-change via the process $\lambda$.
	
	The time-change processes involved are of the form 
	\begin{equation*}
		\Lambda_t(\omega)=\int_0^t\lambda_s(\omega)ds, \quad (t,\omega)\in[0,T]\times \Omega,
	\end{equation*}
	$(T>0)$. Thus the driving noises (which include jumps) are actually beyond the Brownian and the pure Lévy framework. We abandon noises with independent increments and effectively deal with quite general but still treatable martingales.

	Our goal is to find the optimal control $\hatu$ such that
	\begin{equation}\label{J(u) wrt F}
	J(\hatu)=\sup_{u\in\mathcal{A}^{\bbF}}J(u)=\sup_{u\in\mathcal{A}^{\bbF}}\E\left[\int_0^T F(t,\lambda_t,u(t),X^u(t))dt+G(X^u(T))\right],
	\end{equation}
	among the set $\calA^{\bbF}$ of admissible $\bbF$-adapted controls, where 
    $\mathbb{F}=\set{\mathcal{F}_t, \ \t}$ represents the smallest right-continuous filtration generated by $\mu$.
	
	Optimization problems such as \eqref{X_t introduction}, \eqref{J(u) wrt F} are studied, e.g. in \cite{Bonaccorsi, Oksendal1, Oksendal_Forward}.
	In \cite{Oksendal1, Oksendal_Forward}
	the authors present also a sufficient maximum principle and the dynamics include jumps making use of Malliavin calculus. However, being the restrictions on the domain of the Malliavin derivative extremely serious in the context of optimal control, the authors have lifted the study into the white noise framework and work with the Hida-Malliavin calculus on the space of stochastic distributions. The Hida-Malliavin calculus is taylored for Brownian  and for centered Poisson random noises, hence this approach cannot be taken in our work since our driving noises are \emph{not} of the required nature.
	On the other hand, in \cite{Bonaccorsi}, the authors propose a backward SDE approach to solve \eqref{J(u) wrt F}. This is possible due to the introduction of memory in \eqref{X_t introduction} by means of convolution with a completely monotone kernel which allows for a Markovian representation of the solution of \eqref{X_t introduction}.  
	
	Note that a Malliavin/Skorokhod calculus extension to noises with conditional independent increments, is proposed in \cite{DS-Poisson} and \cite{Yablonski}.
	By this, however, we cannot solve the critical issue of the natural restriction of the domains of the involved operators and a Hida-Malliavin type extension is yet not available in the literature. Our approach is then to make use of the \emph{non-anticipating} (NA) derivative. The NA derivative, introduced in \cite{DiNunnoDerivatives} for general martingales and then extended to martingale random fields in \cite{DiNunno2}, is the dual of the Itô integral and has an explicit representation in terms of limit of simple integrands in the Itô framework. Also, the NA derivative provides explicit stochastic integral representations. We stress that, contrarily to the Malliavin derivative, the domain of the NA derivative is the whole $L^2(dP)$, thus not creating problems in the context of optimal controls. To the best of our knowledge this is the first time that the non-anticipating derivative is used in optimal control problems such as \eqref{J(u) wrt F}.
	
	Our approach to the optimization problem \eqref{J(u) wrt F} is based on the analysis of the noise and the information flows associated. Indeed, we observe that there are two filtrations of interest. The first one is the already mentioned $\bbF$ and the second is the filtration $\mathbb{G}:=\set{\mathcal{G}_t, \ \t}$, where $\calG_t:=\F_t \vee \F^{\Lambda}$ generated by $\mu$ and the entire history $\F^\Lambda$ of the time-change processes.
	Note that while $\F_0$ is substantially trivial, $\calG_0=\F^\Lambda$.
	We can regard $\bbG$ as the initial enlargement of $\bbF$ or, we can see $\mathbb{F}$ as partial information with respect to $\mathbb{G}$. With this observation in hands, we work out the solution to problem \eqref{J(u) wrt F} as an optimization problem under partial information. In this we have taken inspiration from \cite{Oksendal3}, where the concept of partial information is however not associated to the properties of the noise, and from \cite{DiNunno1}, where the dynamics are however not of Volterra type. Also, for completeness, we show that our techniques provide necessary and sufficient conditions for the optimization problem
	
	\begin{equation}\label{J(u) wrt G}
	J(\hatu)=\sup_{u\in\mathcal{A}^{\bbG}}J(u)=\sup_{u\in\mathcal{A}^{\bbG}}\E\left[\int_0^T F(t,\lambda_t,u(t),X^u(t))dt+G(X^u(T))\right]
	\end{equation}
	on the set $\calA^\bbG$ of admissible $\bbG$-adapted controls, where $\calA^{\bbF}\subset \calA^{\bbG}$. 
	
	The study of maximum principles is associated to a stochastic Hamiltonian map of the so-called dual variables, which in turn are obtained from the solution of a backward stochastic equation. In the sequel, we deal with backward stochastic differential equations (BSDEs) of type 	
	\begin{equation}\label{BSVIE_Intro}
		p(t)=\xi(t)+\int_t^T g(s,\lambda_s,p(s\shortminus),q(s,\cdot))ds-\int_t^T\!\!\!\int_{\mathbb{R}} q(s,z)\mu(dsdz),
	\end{equation}  
	under the filtration $\bbG$. Notice that, these backward equations are not of Volterra type. This is because our Hamiltonian functional is going to involve also the NA-derivatives of the adjoint process $p$. A different approach could have been to follow the work in \cite{Oksendal_Forward}, where the authors deal with a backward stochastic Volterra integral equation (BSVIE) of the form:
	\begin{equation*}
	    p(t)=\xi(t)+\int_t^T g(s,\lambda_s,p(s\shortminus),q(t,s,\cdot))ds-\int_t^T\!\!\!\int_{\mathbb{R}} q(t,s,z)\mu(dsdz).
	\end{equation*}
	Even though this approach would allow us to work with simpler Hamiltonian functionals (in the sense that the NA-derivative of $p(t)$ would not be involved) we would need to assume smoothness conditions with respect to $t$ on $q(t,s,z)$ and, to the best of our knowledge, is not clear to what extent those properties are satisfied.
	
	\noindent Existence and uniqueness of \eqref{BSVIE_Intro} can be retrieved from \cite{DiNunno1}. The study of the BSDE under $\bbG$ is in itself critically based on the stochastic integral representation in the form 
	\begin{equation}\label{stochastic integral representation}
		\xi=\xi^{0}+\int_0^T\!\!\!\int_\R \phi(s,z) \mu(dsdz),
	\end{equation}
	where $\xi^{0}$ is $\calG_0$-measurable and the integrand $\phi$ is $\bbG$-predictable. These results are readily available in terms of their existence in the classical Kunita-Watanabe Theorem, while the \emph{explicit} form of $\phi$ is given by means of the NA derivative in \cite{DiNunnoDerivatives} Theorem 3.1 and \cite{DiNunno2} Theorem 3.1.

The paper is organized as follows. In the next section we give a presentation of the framework providing the necessary details for the random measure $\mu$ and the information flows that we are going to use. In Section 3 we prove a sufficient maximum principle and in Section 4 the corresponding necessary maximum principle. Lastly, we show how the results obtained can be applied to characterise the solution in the optimal harvesting problem associated to the dynamics \eqref{X(t) example}.

	\section{The noise and the non-anticipating derivative}
	
	Let us consider a complete probability space $\probspace$ and a time horizon $T<\infty$. We shall consider the noise on the time-space $$\bbX:=[0,T]\times \R:	=\Big([0,T]\times \set{0}\Big)\cup\Big([0,T]\times \R_0\Big),$$ where $\R_0=\R\backslash \set{0}$. The Borel $\sigma$-algebra on $\bbX$ is denoted $\mathcal{B}_\bbX$. Let $\calL$ be the space of the two dimensional stochastic processes $\lambda=(\lambda^B,\lambda^H)$ such that, for each component $k=B,H$, we have that
	\begin{enumerate}
		\item $\lambda_t^k\geq 0 \ P-a.s.$ for all $\t$,
		\item $\lim_{h\rightarrow 0}P\left(|\lambda_{t+h}^k-\lambda_t^k|\geq \epsilon \right)=0$ for all $\epsilon>0$ and almost all $\t$,
		\item $\E\left[\int_0^T \lambda_t^k dt\right]<\infty$.
	\end{enumerate}
	The processes $\lambda\in\calL$ represent the \emph{stochastic time-change rate}. 
	Let $\nu$ be a $\sigma$-finite measure on the Borel sets of $\R_0$ satisfying $\int_{\R_0}z^2\nu(dz)<\infty$.
	We define the random measure $\Lambda$ on $\mathcal{B}_\bbX$ by
	\begin{equation}\label{defLambda}
	\Lambda(\Delta):=\int_0^T\mathds{1}_{\set{(t,0)\in\Delta}}(t)\lambda_t^B dt+\int_0^T \!\!\!\int_{\R_0}\mathds{1}_\Delta(t,z)\nu(dz)\lambda_t^H dt, \quad \Delta \subseteq \bbX.
	\end{equation}
	Furthermore, denote the restrictions of $\Lambda$ to $[0,T]\times\set{0}$ and $[0,T]\times \R_0$ by  $\Lambda^B$ and $\Lambda^H$, respectively.	
	For later use we also introduce the filtration
	\begin{equation*}
	\bbF^\Lambda=\set{\F_t^\Lambda, \t},
	\end{equation*}
	where $\F_t^\Lambda$ is generated by the values of $\Lambda$ on the Borelian sets of $[0,t]\times\R$. Set $\F^\Lambda:=\F_T^\Lambda$.
	We recall the following definitions.
	\begin{defin}\label{definitionBH}
		The \emph{conditional Gaussian measure} (given $\F^\Lambda$) $B$ is a signed random measure on the Borel sets of $[0,T]\times \set{0}$ satisfying
		\begin{itemize}
			\item[A1.] $P\left(B(\Delta)\leq x |\F^\Lambda\right)=P\left(B(\Delta)\leq x |\Lambda^B(\Delta)\right)=\Phi\left(\frac{x}{\sqrt{\Lambda^B(\Delta)}}\right), \newline x \in \R, \Delta\subseteq [0,T]\times \set{0}$. Here $\Phi$ is the cumulative probability distribution function of a standard normal random variable.
			\item[A2.] For all disjoint  $\Delta_1, \Delta_2 \subseteq [0,T]\times \set{0}$, $B(\Delta_1)$ and $B(\Delta_2)$ are conditionally independent given $\F^\Lambda$.
		\end{itemize}
		The \emph{conditional Poisson measure} (given $\F^\Lambda$) $H$ is a random measure on the Borel sets of $[0,T]\times \R_0$ satisfying 
		\begin{itemize}
			\item[A3.] $P\left(H(\Delta)\!=\! k |\F^\Lambda\right)\!=\!P\left(H(\Delta)\!=\!k |\Lambda^H(\Delta)\right)\!=\!\frac{\Lambda^H(\Delta)^k}{k!}e^{-\Lambda^H(\Delta)}$, $k \in \N$, $\Delta\subseteq [0,T]\times \R_0$.
			\item[A4.] For all disjoint  $\Delta_1, \Delta_2 \subseteq [0,T]\times \set{\R_0}$, $H(\Delta_1)$ and $H(\Delta_2)$ are conditionally independent given $\F^\Lambda$.
		\end{itemize}
			Moreover,
	\begin{itemize}
		\item[A5.] $B$ and $H$ are conditionally independent given $\F^\Lambda$.
	\end{itemize}
	Also the\emph{ conditional centered Poisson random measure} is defined  as $$\tH(\Delta):=H(\Delta)-\Lambda^H(\Delta), \quad \Delta \subset \bbX.$$
	Observe that if $\lambda^B$ and $\lambda^H$ were deterministic, then $B$ would be a Gaussian process and $H$ a Poisson random measure. Furthermore, $B$ would be a Wiener process if $\lambda^B\equiv1$ and $H$ a homogeneous Poisson random measure for $\lambda^H\equiv 1$.
		
	\end{defin} 
	\begin{defin}
		We define the signed random measure $\mu$ on the Borel sets $\Delta \subseteq \bbX$ by
		\begin{equation*}
		\mu(\Delta):=B\left(\Delta\cap[0,T]\times\set{0}\right)+\tH\left(\Delta\cap[0,T]\times\R_0\right).
		\end{equation*}
		The random measure $\mu$ has \emph{conditionally independent values}, see \cite{Grigelionis,Serfozo}.
			Observe that (A1) and (A3) yield 
		\begin{equation}\label{martingalemu}
			\E[\mu(\Delta)|\F^\Lambda]=0, \qquad \E[\mu(\Delta)^2|\F^\Lambda]=\Lambda(\Delta) , \quad \Delta \subseteq \bbX.
		\end{equation}
	\end{defin}
	The random measures $B$ and $H$ are related to a time-changed Brownian motion and time-changed pure jump Lévy process. 
		To illustrate, consider the processes on $[0,T]$:
		\begin{align*}
			B_t&:=B([0,t]\times\set{0}), &\Lambda_t^B:=\int_0^t\lambda_s^B ds, \ \ \  \\
			\eta_t&:=\int_0^t\!\!\!\int_{\R_0} z\tH(dsdz),  &\Lambda_t^H:=\int_0^t\lambda_s^H ds, \quad  
		\end{align*} 
		and compute the characteristic functions of $B$ and $\eta$. From (A1) and (A3) we have that
		\begin{align*}
		\E\left[e^{icB_t}\right]&=\int_\R\E\left[e^{icB_t}|\Lambda_t^B=x\right]P_{\Lambda_t^B}(dx)=\int_\R e^{\frac 1 2 c^2 x}P_{\Lambda_t^B}(dx), \quad c\in\R,
		\end{align*}
		 where $P_{\Lambda_t^B}$ is the probability distribution of the time-change  $\Lambda_t^B$. Correspondingly, we have that
		\begin{align*}
		\E\left[e^{ic\eta_t}\right]
		&=\int_\R\exp \left\{\int_{\R_0}[e^{iczx}-1-iczx]\nu(dz)  \right\}P_{{\Lambda}_t^H}(dx), \quad c\in\R,
		\end{align*}
		where $P_{{\Lambda}_t^H}$ is the probability distribution of the time-change ${\Lambda}_t^H$. Indeed we recall the following characterization \cite[Theorem 3.1]{Serfozo} :

	\begin{teo}
		Let $W_t$, $\t$, be a Brownian motion independent of $\Lambda^B$ and $N_t$, $\t$, be a centered pure jump Lévy process with Lévy measure $\nu$ independent of $\Lambda^H$. Then $B$ satisfies (A1)-(A2) if and only if, for any $t\geq 0$, $B_t\stackrel{d}{=}W_{\Lambda_t^B}$ and $\eta$ satisfies (A3)-(A4) if and only if, for any $t\geq 0$, $\eta_t\stackrel{d}{=}N_{{\Lambda}_t^H}$.
	\end{teo}
	
	In the sequel we shall consider two types of information flows. The first one is represented by the filtration
	\begin{equation*}
		\bbF:=\set{\F_t, \ \t},\qquad\F_t:=\bigcap_{r>t}\F_r^\mu,
	\end{equation*}
	where $\bbF^\mu:=\set{\F_t^\mu, \ \t}$ is generated by the values $\mu(\Delta)$,  $\Delta\subset[0,t]\times \R$, $\t$. Correspondingly, let $\bbF^B:=\set{\F_t^B, \t}$ denote the filtration generated by $B(\Delta\cap[0,t]\times\{0\})$, and $\bbF^H:=\set{\F_t^H, \t}$ the filtration generated by $H(\Delta\cap[0,t]\times\R_0)$. We remark that, for any $\t$, $\F_t^\mu=\F_t^B\vee\F_t^H\vee\F_t^\Lambda$. See \cite{DS-Poisson}.

	The second information flow of interest is $$\bbG:=\set{\calG_t, \ \t}, \quad \calG_t:=\F_t^\mu\vee\F^\Lambda.$$ The filtration $\bbG$ is right-continuous, see \cite{DiNunno1}. Moreover we note that $\calG_T=\F_T$, $\calG_0=\F^\Lambda$, and $\F_0$ is substantially trivial. Namely, $\bbG$ includes information on the future values of $\Lambda^B$ and $\Lambda^H$. In the sequel we shall technically exploit the interplay between the two filtrations.

	For $\Delta \subseteq (t,T]\times \R$, the conditional independence in (A2) and (A4), together with \eqref{martingalemu} yield
	\begin{equation}\label{Martingale property}
		\E[\mu(\Delta)|\calG_t]=\E[\mu(\Delta)|\F_t\vee \F^\Lambda]=\E[\mu(\Delta)|\F^\Lambda]=0.
	\end{equation}
	Moreover, (A5) gives us
	\begin{equation*}
		\E[\mu(\Delta_1)\mu(\Delta_2)|\calG_t]=\E[\mu(\Delta_1)|\F^\Lambda]\E[\mu(\Delta_2)|\F^\Lambda]=0,
	\end{equation*}
	for disjoint $\Delta_1,\Delta_2\subseteq (t,T]\times\R$. Hence, $\mu$ is a martingale random field with respect to $\bbG$, see e.g. \cite{DiNunno2} Definition 2.1: 
	\begin{defin}
		A square integrable \emph{martingale random field $\mu$ with conditionally orthogonal values} is a stochastic set function $\mu(\Delta)$, $\Delta\subseteq \bbX$ such that
		\begin{itemize}
			\item $m(\Delta):=\E[\mu(\Delta)^2]=\E[\Lambda(\Delta)], \ \Delta \subseteq \bbX$, defines a variance measure
			\item $\mu$ is $\bbG$-adapted
			\item $\mu$ satisfies the martingale property \eqref{Martingale property}
			\item $\mu$ has conditionally orthogonal values: $			\E[\mu(\Delta_1)\mu(\Delta_2)|\calG_t]=0$, for every disjoint $\Delta_1,\Delta_1\in(t,T]\times \R$.
		\end{itemize}
	It is immediate to see that $\mu$ is also a martingale random field with respect to $\bbF$.
	\end{defin}
	With the above structures, we access the framework of Itô stochastic integration. For this we introduce $\mathcal{I}^\bbG\subseteq L^2(d\Lambda \times dP)$ representing the subspace of the random fields admitting a $\bbG$-predictable modification and $\mathcal{I}^{\bbF}\subset \mathcal{I}^{\bbG}$, the one of $\bbF$-predictable random fields. Observe that, for all $\phi \in\mathcal{I}^\bbG$, we have that
		\begin{equation}\label{QuadVar}
		\E\bigg[\left(\iint_\bbX \phi(s,z) \mu(dsdz)\right)^2\bigg]=\E\bigg[\iint_\bbX \phi(s,z)^2\Lambda(dsdz)\bigg]
		\end{equation}
	thanks to (A5) and the martingale property of $\mu$.
		
		In this work we shall make use of the non-anticipating derivative introduced in \cite{DiNunno2} for martingale random fields.

	\begin{defin}
		The \emph{non-anticipating derivative} (NA-derivative) $\calD$ is a linear operator defined for \emph{all} the elements $\zeta \in L^2(dP)$ as the limit in $L^2(d\Lambda\times dP)$
		\begin{equation}\label{Definition_Derivative}
			\calD \zeta:=\lim_{n\rightarrow\infty}\varphi_n,
		\end{equation}
	of simple $\bbG$-predictable random fields $\varphi_n$, $n\in\N$, defined as:
		\begin{equation*}
			\varphi_n(t,x):=\sum_{k=1}^{K_n}\E\left[\zeta\frac{\mu(\Delta_{nk})}{\E[\Lambda(\Delta_{nk})|\calG_{s_{nk}}]}\middle|\calG_{s_{nk}}\right]\mathds{1}_{\Delta_{nk}}(t,x), \quad (t,x)\in\bbX.
		\end{equation*}
		Here the Borel sets $\Delta_{nk}$ take the form $\Delta_{nk}:=(s_{nk},u_{nk}]\times B_{nk}$, $k=1,...,K_n$, with $0\leq s_{nk}\leq u_{nk}\leq T$, and $B_{nk}\in\mathfrak{B}$ where $\mathfrak{B}$ is any countable semi-ring that generates the Borel $\sigma$-algebra $\mathcal{B}(\R)$. Then $\bigcup_{n\in\N}\bigcup_{k=1}^{K_n}\Delta_{nk}=\bbX$. With a slight abuse of terminology we call the sets $\Delta_{nk}$, $k=1,...,K_n$, a \emph{partition} of $\bbX$ with refinement $n$. Clearly all the sets $\Delta_{nk}$, $k=1,...,K_n$, $n\in\N$ constitute a semiring generating $\mathcal B(\bbX)$ see, e.g. \cite{DiNunno2} and the references therein.
	\end{defin}
	The NA-derivative allows for an explicit integral representation. Namely the integrand is characterized in terms of the inputs: the very random variable to represent, the integrator, and the filtration. See Theorem 3.1 in \cite{DiNunno2}.

	\begin{teo}\label{C-Ok-nonanticipating}
		For any $\xi\in L^2(dP)$ the NA-derivative $\calD \xi$ is well defined and the following stochastic integral representation holds
		\begin{equation}\label{C-O_non_ant}
		\xi=\xi^{0}+\iint_\bbX\calD_{t,z} \xi  \ \mu(dtdz),
		\end{equation}
		where $\xi^{0}=\E\left[\xi|\F^\Lambda\right]$ satisfies $\calD\xi^{0}\equiv 0$. 
	\end{teo}
	\noindent The existence and unicity of a stochastic integral representation is well-known from the Kunita-Watanabe Theorem. Theorem \ref{C-Ok-nonanticipating} provides an explicit representation to the integrand. The spirit of this result is in line with representations à la Clark-Haussman-Ocone (CHO), see, e.g. \cite{DOP-book}. However in that case the noise is either a Brownian motion or a centered Poisson random measure and the integrand is characterized in terms of the Malliavin derivative. We remark that an extension of the Malliavin calculus and CHO representations to the conditional Brownian and the conditional Poisson cases is provided in \cite{Yablonski} and \cite{DiNunno1}. When applying Malliavin calculus to optimal control, the domain of the Malliavin derivative constitutes a serious restriction as the variables depend on a control yet to be found. In \cite{Oksendal2} this was overcome for the Brownian and centered Poisson cases by using the Hida-Malliavin extension which is an extension of Malliavin calculus to the white noise framework (stochastic distributions), see \cite{DOP-book}. At present there is no such an extension for time-changed noises hence the method cannot be used. In this paper we suggest to use the NA-derivative, which has no restrictions on the domain and it is well defined for all martingales in $L^2(dP)$ as integrators. Furthermore, from Theorem \ref{C-Ok-nonanticipating} we can see that $\calD$ is actually the dual of the Itô integral:
	
	\begin{prop}\label{Yab} For all $\phi$ in $\mathcal{I}^{\bbG}$ and all $\xi$ in $L^2(dP)$, we have
		\begin{equation*}
		\E\left[\xi\iint_\bbX\phi(t,z)\mu(dtdz)\right]=\E\left[\iint_\bbX q(t,z)\calD_{t,z}\xi \ \Lambda(dtdz)\right].
		\end{equation*}
	\end{prop}
	\noindent Also we have the martingale representation theorem:
	\begin{teo}\label{martingalerepresentation}
	For any square integrable $\bbG$ martingale, $M(t), \t,$ the following representation holds true
	\begin{equation*}
	M(t)=\E[M(T)|\F^\Lambda]+\int_0^t\!\!\!\int_\R\calD_{s,z}M(T)\mu(dsdz).
	\end{equation*}
	\end{teo}
	For future use we also introduce the space $S$ of the $\bbG$-adapted stochastic processes $p(t,\omega)$, $\t$, $\omega \in \Omega$ such that
	\begin{equation*}
	    \|p\|_S:=\E\left[\sup_{0\le t \le T}|p(t)|^2\right]^{1/2}<\infty.
	\end{equation*}

\section{A sufficient maximum principle with time-change}
\label{SecSuff}
	We are now ready to study the optimization problem \eqref{J(u) wrt F} with performance functional
	\begin{equation}\label{performancefunctionalG}
	J(u)=\E\left[\int_0^T F(t,\lambda_t,u(t),X^u(t))dt+G(X^u(T))\right],
	\end{equation} 
	where
	\begin{align*}
		F&:[0,T]\times [0,\infty)^2\times \calU\times \R\times \Omega \longrightarrow \R,\\
		G&:\R\times \Omega \longrightarrow \R,
	\end{align*}
	with $\calU$ a closed convex subset of $\R$. For all $\lambda \in [0,\infty)^2$, $u\in\calU$, $x\in\R$ the process $F(\cdot,\lambda,u,x,\cdot)$ is $\bbF$-adapted and the mapping $F(t,\lambda,u,x)$ is $C^1$ in $x$ $P$-a.s. uniformly w.r.t. $\t$, $\lambda \in [0,\infty)^2$, $u\in\calU$. Also for all $x\in\R$, $G(x,\cdot)$ is $\F_T$-measurable and $G$ is $C^1$ in $x$ $P$-a.s. uniformly w.r.t. $\t$, $\lambda \in [0,\infty)^2$, $u\in\calU$. The controlled dynamics of $X$ are given by the equation
	\begin{equation}\label{dXt}
	X^u(t)=X_0+\int_0^t b(t,s,\lambda_s,u(s),X^u(s\shortminus))ds +\int_0^t\!\!\!\int_{\R} \kappa(t,s,z,\lambda_s,u(s),X^u(s\shortminus))\mu(dsdz),
	\end{equation}
	where $X_0\in\R$ and the coefficients are given by the mappings
	\begin{align*}
	b&:[0,T]\times[0,T]\times[0,\infty)^2\times\calU\times\R\times \Omega\longrightarrow\R, \\
	\kappa&:[0,T]\times[0,T]\times\R\times[0,\infty)^2\times\calU\times\R\times \Omega\longrightarrow \R.
	\end{align*}
	We assume $b(t,\cdot,\lambda,u,x,\cdot)$ and $\kappa(t,\cdot,z,\lambda,u,x,\cdot)$ to be $\bbF$-predictable for all $\t,\lambda\in[0,\infty)^2,u\in \calU,x\in\R$ and $z\in \R$. We also require them to be $C^2$ with respect to $t$ and to $x$ with partial derivatives $L^2$-integrable with respect to $dt\times dP$ and $d\Lambda\times dP$, respectively. Notice also that we will often drop the superscript $u$ when it is clear the dependence of $X$ on $u$.

 Later on we can see the coefficients $b$ and $\kappa$ in a functional setup:
\begin{align*}
	b&:[0,T]\times[0,T]\times\Xi_{\R^2_+}\times \Xi_{\calU}\times\Xi_\R\times \Omega\longrightarrow\R, \\
	\kappa&:[0,T]\times[0,T]\times \R \times\Xi_{\R^2_+}\times \Xi_{\calU}\times\Xi_\R\times \Omega\longrightarrow\R,
	\end{align*}
 where we denoted by $\Xi_S$ the space of measurable function on $[0,T]$ with values in $S$. Then we can interpret the coefficients in \eqref{dXt} via the evaluation at the point $s\in[0,T]$:
 \begin{align*}
     b(t,\cdot,\lambda_\cdot, u(\cdot), X^u(\cdot))(s) &= b(t,s,\lambda_s,u(s),X^u(s\shortminus))\\
    \kappa(t,\cdot,z,\lambda_\cdot, u(\cdot), X^u(\cdot))(s) &= \kappa(t,s,z,\lambda_s,u(s),X^u(s\shortminus)).
 \end{align*}
 We assume that $b$ and $\kappa$ are Fréchet differentiable (in the standard topology of càdlàg paths) with $C^2$ regularity in $t$, $x$ and $u$ (with the corresponding derivatives).

	In the sequel we assume existence and uniqueness of a solution for \eqref{dXt}. Sufficient conditions for this are provided in the next result, which is in line with the study in \cite{Oksendal_Forward}, though there the driving noises are the Brownian motion and Poisson random measure.
	
	\begin{teo}\label{Esistenza Forward}
		Assume that:
		\begin{enumerate}
			\item $b(t,\cdot,\lambda,u,x,\cdot)$ and $\kappa(t,\cdot,z,\lambda,u,x,\cdot)$ are $\bbF$-predictable for all $\t,z\in \R,\lambda\in[0,\infty)^2,u\in \calU$ and $x\in\R$.
			\item $b(t,s,\lambda,u,\cdot)$ and $\kappa(t,s,\cdot,\lambda,u,\cdot)$ are Lipschitz continuous with respect to $x$, uniformly in $t,s\in[0,T]^2$, $u\in\calU$, $\lambda\in[0,\infty)^2$, i.e., for all $x_1,x_2\in\R$,
			\begin{align*}
				&|b(t,s,\lambda,u,x_1)-b(t,s,\lambda,u,x_2)|+|\kappa(t,s,0,\lambda,u,x_1)-\kappa(t,s,0,\lambda,u,x_2)|\sqrt{\lambda^B}\\&+\left(\int_{\R_0}|\kappa(t,s,z,\lambda,u,x_1)-\kappa(t,s,z,\lambda,u,x_2)|^2\nu(dz)\right)^{1/2}\sqrt{\lambda^H}\leq C|x_1-x_2|, \quad P-a.s.
			\end{align*}
			\item $b(t,s,\lambda,u,\cdot)$ and $\kappa(t,s,z,\lambda,u,\cdot)$ have linear growth with respect to $x$, i.e., for all $t,s\in[0,T]^2$, $u\in\calU$, $\lambda\in[0,\infty)^2$, $x\in\R$, we have
			\begin{align*}
				|b(t,s,\lambda,u,x)|&+|\kappa(t,s,0,\lambda,u,x)|\sqrt{\lambda^B} \\&+\left(\int_{\R_0} |\kappa(t,s,z,\lambda,u,x)|^2\nu(dz)\right)^{1/2}\sqrt{\lambda^H}\leq C(1+|x|) \quad P-a.s.
			\end{align*}
		\end{enumerate}
	Then there exists a unique $\bbF$-adapted solution to \eqref{dXt} in $L^2(dt\times dP)$. 
	\end{teo}
\begin{proof}
	The proof follows a classical Picard iteration scheme.  Here we provide the main ideas.	Fix $u\in\calA^\bbF$ and define inductively
	\begin{align*}
		X^0(t)&:=X_0\\
		X^n(t)&:=X_0+\int_0^tb(t,s,\lambda_s,u(s),X^{n-1}(s))ds\\
		&\quad +\int_0^t\!\!\!\int_\R\kappa(t,s,z,\lambda_s,u(s),X^{n-1}(s\shortminus))\mu(dsdz), \qquad \t, \ n\geq1
	\end{align*}
	Then, for all $\t$ and for all $n\geq 1$, we have the following estimate
	\begin{align*}
		\E&\left[|X^{n+1}(t)-X^n(t)|^2\right]\leq 2\E\left[t\int_0^t|b(t,s,\lambda_s,u(s),X^{n}(s\shortminus))-b(t,s,\lambda_s,u(s),X^{n-1}(s\shortminus))|^2ds\right]\\
		&+2\E\left[\int_0^t\!\!\!\int_\R|\kappa(t,s,z,\lambda_s,u(s),X^{n}(s\shortminus))-\kappa(t,s,z,\lambda_s,u(s),X^{n-1}(s\shortminus))|^2\Lambda(dsdz)\right].
	\end{align*}
	By \eqref{defLambda} and using the Lipschitz condition on $b$ and $\kappa$, we get
	\begin{equation*}
		\E\left[|X^{n+1}(t)-X^n(t)|^2\right]\leq 2 C^2 \E\left[t\int_0^t2|X^n(s\shortminus)-X^{n-1}(s\shortminus)|ds\right],
	\end{equation*}
	which leads to
	\begin{equation}\label{FORWARD(stima 1)}
		\E\left[|X^{n+1}(t)-X^n(t)|^2\right]\leq K\E\left[\int_0^t|X^{n}(s\shortminus)-X^{n-1}(s\shortminus)|^2ds\right],
	\end{equation}
	 for $K:=4TC^2$. Also, by the linear growth condition \textit{3.} on $b$ and $\kappa$, we get that
	\begin{equation}\label{FORWARD(stima 2)}
		\E\left[|X^1(t)-X^0(t)|\right]\leq Kt(1+X_0)^2.
	\end{equation}
	Combining now \eqref{FORWARD(stima 1)} and \eqref{FORWARD(stima 2)}, we have that
	\begin{equation*}
		\E\left[|X^{n+1}(t)-X^n(t)|^2\right]\leq \frac{2(1+X_0)^2(Kt)^{n+1}}{(n+1)!}.
	\end{equation*}
Thus we have that $\{X(t)^n\}_{n=1}^{\infty}$ is a Cauchy sequence in $L^2(dP)$ and $\{X(t)^n\}_{n=1}^{\infty}$ is in $L^2(dP\times dt)$ 
Taking the limit on $n\rightarrow\infty$  gives the solution to \eqref{dXt}.
	 The uniqueness is obtained by standard arguments and estimates similar to the ones above.
\end{proof}

    	Before moving forward, we need to state a fundamental result that will allow us to rewrite $X$ in \eqref{dXt} in differential form. This is due to \cite{Protter} and it is known as transformation rule. Hereafter we state the result within our setting.
    	
    	\begin{lema}(Transformation rule)\label{Transformation rule} 	Assume that for all $z\in\R$, $\lambda \in [0,\infty)^2$, $u\in\calU$, $x\in\R$ the partial derivative of $\kappa$ with respect to $t$ (denoted with $\partial_t \kappa(t,s,z,\lambda,u,x)$) is locally bounded (uniformly in $t$) and satisfies
	\begin{equation}\label{lipschitz_partial_t}
	    |\partial_t\kappa(t_1,s,z,\lambda,u,x)- \partial_t \kappa(t_2,s,z,\lambda,u,x)|\leq K |t_1-t_2|,
	\end{equation}
	for some $K>0$ and for each fixed $s\leq t$, $\lambda \in [0,\infty)^2$, $u\in\calU$, $x\in\R$.
    	
    Then, the forward equation \eqref{dXt} can be rewritten in differential notation as 
    \begin{align}
        dX(t) &= \Big(b(t,t,\lambda_t,u(t),X(t)) + \int_0^t\partial_t b(t,s, \lambda_s,u(s),X(s))ds \nonumber \\
        &\quad+ \int_0^t\!\!\!\int_\R\partial_t\kappa(t,s,z,\lambda_s,u(s),X(s))\mu(dsdz)\Big)dt + \int_\R\kappa(t,t,z, \lambda_t,u(t),X(t))\mu(dtdz)\label{functionalSDE_orig}.
    \end{align}
    \end{lema}
    \begin{proof}
        The proof follows the one in \cite{Protter}. We report it here for completeness. Observe that
        \begin{align*}
            X(t) &= \int_0^t b(t,s,\lambda_s,u(s),X(s)) ds + \int_0^t\!\!\!\int_\R\kappa(t,s,z,\lambda_s,u(s),X(s))\mu(dsdz)\\
            &=\int_0^tb(t,s,\lambda_s,u(s),X(s))ds +\int_0^t\!\!\!\int_\R\kappa(s,s,z,\lambda_s,u(s),X(s)) \mu(dsdz) \\&\quad + \int_0^t\!\!\!\int_\R \kappa(t,s,z,\lambda_s,u(s),X(s))-\kappa(s,s,z,\lambda_s,u(s),X(s))\mu(dsdz) 
        \end{align*}
        Note that 
        \begin{align*}
            \kappa(t,s,z,\lambda_s,u(s),X(s))-\kappa(s,s,z,\lambda_s,u(s),X(s)) 
            &= \int_s^t\partial_r\kappa(r,s,z, \lambda_s, u(s), X(s))dr 
            \\&= \int_0^t\mathds{1}_{s\le r}\partial_r \kappa(r,s,z,\lambda_s,u(s),X(s)) dr
        \end{align*} Then we can apply the Fubini theorem for stochastic integration as in \cite{Jacod} and we obtain that
        \begin{align*}
            \int_0^t\!\!\!\int_\R &\kappa(t,s,z,\lambda_s,u(s),X(s))-\kappa(s,s,z,\lambda_s,u(s),X(s))\mu(dsdz) \\&= \int_0^t\!\!\!\int_\R\left\{\int_0^t\mathds{1}_{s\leq r}\partial_r \kappa(r,s,z,\lambda_s,u(s),X(s)) dr\right\}\mu(dsdz)\\
            &=\int_0^t\left\{\int_0^r\!\!\!\int_\R\partial_r\kappa(r,s,z,\lambda_s,u(s),X(s))\mu(dsdz)\right\}dr.
        \end{align*}
        The well posedness and the Lebesgue integrability of $\int_0^r\!\int_\R\partial_r\kappa(r,s,z,\lambda_s,u(s),X(s))\mu(dsdz)$, $r\in[0,t]$ is achieved in Theorem 3.2 \cite{Protter} thanks to \eqref{lipschitz_partial_t}.
    \end{proof}
	
	\begin{obs} \label{obs FSDEs}
 \emph{(A link with functional SDEs)}
	Lemma \ref{Transformation rule}, suggests a link between the Volterra integral equations of the kind \eqref{dXt} and functional SDEs (FSDEs). It is in fact clear that, by defining
	\begin{align*}
	    B(t,\lambda_\cdot, u_\cdot, X_\cdot, Z_\cdot) &:= \Big(b(t,t,\lambda_t,u(t),X(t)) + \int_0^t\partial_t b(t,s, \lambda_s,u(s),X(s))ds+ Z(t)\Big),
	\end{align*}
	where
	\begin{equation*}
	    Z(t) = \int_0^t\!\!\!\int_\R\partial_t\kappa(t,s,z,\lambda_s,u(s),X(s))\mu(dsdz),
	\end{equation*}
	we have that \eqref{dXt} can be rewritten as
	\begin{equation}\label{functionalSDE}
	X(t) = X_0+\int_0^t B(t,\lambda_\cdot, u_\cdot, X_\cdot, Z_\cdot) dt + \int_0^t\!\!\!\int_\R\kappa(t,t,z,\lambda_t,u(t),X(t))\mu(dtdz).
	\end{equation}
	We notice that \eqref{functionalSDE} is a functional SDE, so we could have tried to state an existence result for functional SDEs instead of using Theorem \ref{Esistenza Forward}.  Some existence results for SDEs such as \eqref{functionalSDE} are available (see e.g. \cite{Cont, PossamaiSDE, soluzioneSDE,DiNunnoBanos, OksendalDhal}), but no one of those deals with noises such as $\mu$. While some of those results (e.g. \cite{Cont, PossamaiSDE, soluzioneSDE}) present condition that would be too restrictive for the current setting, we also point out that the results presented in \cite{OksendalDhal, DiNunnoBanos} could possibly be extended to the current framework. Nonetheless, this would require to impose some Lipschitz and linear growth conditions on $b$ and $\kappa$ (like in Theorem \ref{Esistenza Forward}) and, additionally, to impose a Lipschitzianity condition on $\partial_t b$, not required in the hypothesis of Theorem \ref{Esistenza Forward}.
	\end{obs}

    \vspace{2mm}
    Having discussed the existence of a solution for \eqref{dXt}, we are finally ready to proceed to our optimization results. We start by introducing the notion of admissible controls:
    	\begin{defin}\label{definition_admissible}
		The admissible controls for \eqref{dXt} in the optimization problems \eqref{J(u) wrt F} and \eqref{J(u) wrt G} are predictable stochastic processes  $u:[0,T] \times \Omega \longmapsto \calU$  such that $X$ in \eqref{dXt} has a unique strong solution and
		\begin{equation*}
		    \E\left[\int_0^T F(t,\lambda_t,u(t), X(t))dt+G( X(T))+|\partial_xG(X(T))|^2\right]<\infty
		\end{equation*}
		We denote $\calA^{\bbF}$ and $\calA^{\bbG}$ the sets of $\bbF$- or $\bbG$-predictable controls, respectively. We say that $(\hatu,\hatX)$ is an \emph{optimal pair} if 
		\begin{equation}\label{J(u) wrt both}
			J(\hatu)=\sup_{u\in\calA^{\cdot}}\E\left[\int_0^T F(t,\lambda_t,u(t), X(t))dt+G( X(T))\right],
		\end{equation}
		where $\hat X:=X^{\hatu}$ is as in \eqref{dXt}, and $\calA^{\cdot}$ is either the set $\calA^\bbF$ or $\calA^\bbG$.
	\end{defin}	
	Define $\mathcal{R}^{\bbG}$ to be the space of $\bbG$-predictable processes with values in $L^2(dP)$. We remark that, if $y\in\mathcal{R}^{\bbG}$, then the NA-derivative \eqref{Definition_Derivative} is also in $\mathcal R^\bbG$ i.e. for all $t,z$ $\calD_{t,z}y(\cdot)\in\mathcal R^\bbG$. In the sequel, when no confusion arises, we will denote with $\calD_{t,0}y(\cdot)$ the NA-derivative with respect to the conditional Brownian motion, and with $\calD_{t,z}y(\cdot)$, $z\in\R_0$, the NA-derivative with respect to the conditional Poisson random measure.

 \vspace{2mm}
	In view of the Volterra structure of the dynamics \eqref{dXt}, the system is not Markovian. We tackle the problem \eqref{J(u) wrt both} by the maximum principle approach, better suited in this case, see e.g. \cite{Stochastic_controls}. We introduce the Hamiltonian function:	
	$$\calH:[0,T]\times\Xi_{{\R}^2_+}\times \Xi_{\calU}\times\Xi_{\R}\times\mathcal R^{\bbG}\times \Xi_{\mathcal Z}\times\Omega\longrightarrow\R,$$ as the mapping given by the sum
	\begin{equation}\label{hamiltonG}
	\calH(t,\lambda,u,x,p,q):=H_0(t,\lambda,u,x,p,q)+H_1(t,\lambda,u,x,p,q)
	\end{equation}
	of the two components
	\begin{align*}
	H_0(t,\lambda,u,x,p,q)&:=F(t,\lambda_t,u_t,x_t)+b(t,t,\lambda_t,u_t,x_t)p(t)+\kappa(t,t,0,\lambda_t,u_t,x_t)q_t(0)\lambda_t^B\\
	&\quad +\int_{\R_0}\kappa(t,t,z,\lambda_t,u_t,x_t)q_t(z)\lambda_t^H\nu(dz)\\
	H_1(t,\lambda,u,x,p,q)
	&:=\int_0^t\partial_t b(t,s,\lambda_s,u_s,x_s)ds\ p(t) +\int_0^t\!\!\!\int_\R\partial_t \kappa(t,s,z,\lambda_s,u_s,x_s) \calD_{s,z} p(t)\Lambda(dsdz),
	\end{align*}
	where $\mathcal{Z}$ is the space of functions $q:\R\longrightarrow\R$ such that
    \begin{equation*}
	    |q(0)|^2+\int_{\R_0}q(z)^2\nu(dz)<\infty.
    \end{equation*}

    \begin{obs}
    Following up on Remark \ref{obs FSDEs}, instead of considering \eqref{dXt} as a Volterra equation, we could have taken the FSDE \eqref{functionalSDE} and, following e.g. \cite{OksendalDhal}, write the Hamiltonian functional for the functional SDE. We notice that, regardless of the chosen approach, we would still end up with the Hamiltonian functional \eqref{hamiltonG}.
    \end{obs}

    \vspace{2mm}
	\noindent Associated to $\calH$ \eqref{hamiltonG}, we introduce a BSDE of the type \eqref{BSVIE_Intro}, which we study under $\bbG$:
	\begin{align}
	p(t)=&\partial_x G(X(T))+\int_t^T\partial_x\calH(s,\lambda_{s},u(s),X(s\shortminus),p(s\shortminus),q(s,\cdot))ds\nonumber\\
	&-\int_t^T\!\!\!\int_{\R}q(s,z)\mu(dsdz), \quad \t,	\label{dYt}
	\end{align}
where the derivative $\partial_x\calH$ is meant in the Fréchet sense.
    
    \noindent Sufficient conditions to guarantee the existence of \eqref{dYt} on $\mathcal R^{\bbG}\times \mathcal{I}^{\bbG}$ can be found in \cite{DiNunno1}.
	\begin{obs}
	    Notice that \eqref{dYt} is actually a BSDE and not a Volterra-type backward SDE. In fact, the term $\partial_x H_1(t, \lambda, u, X, p,q)$ in the driver $\partial_x\calH(t, \lambda, u, X, p,q)$, corresponds to
	    \begin{align*}
	   &\partial_xH_1(s,\lambda_s,u,X,p,q) = \partial_x\int_0^s\partial_s b(s,r,\lambda_r,u(r),X(r))dr\ p(s) \\&+ \partial_x\int_0^s\!\!\!\int_\R\partial_s\kappa(s,r,z,\lambda_r,u(r),X(r))\calD_{r,z}p(s)\Lambda(drdz),
	    \end{align*}
   which is a function of time $s$, after integration.
	\end{obs}
	
	The optimal control problem \eqref{J(u) wrt F}:
	\begin{equation}\label{performancefunctionalF}
	J(\hatu)=\sup_{u\in\calA^{\bbF}}J(u)=\sup_{u\in\calA^\bbF}\E\left[\int_0^T F(t,\lambda_t,u(t),X^u(t))dt+G(X^u(T))\right],
	\end{equation} 
	associated to the performance functional \eqref{performancefunctionalG} is treated in the framework of optimization  under partial information. This is inspired by \cite{DiNunno1}, where this approach is taken for standard time-changed dynamics. In the Volterra case treated in the present work, the functionals stemming out of \eqref{hamiltonG} are very different from the ones in \cite{DiNunno1}. Indeed we introduce the mapping $\calH^\bbF$ defined for $\t$, $\lambda \in \Xi_{\R^2_+}$, $u\in\Xi_{\calU}$, $x\in\Xi_\R$, $p\in \mathcal R^\bbG$ and $q\in \mathcal{I}^\bbG$ as
	\begin{align}
	\calH^\bbF(t,\lambda,u,x,p,q)&:=H_0^\bbF(t,\lambda,u,x,p,q)+H_1^\bbF(t,\lambda,u,x,p,q)\nonumber \\&:=\E\left[\calH(t,\lambda,u,x,p,q)|\F_t\right] \label{hamiltonF},
	\end{align}
	where
	\begin{align*}
	H_0^\bbF(t,\lambda,u,x,p,q)&:=F(t,\lambda_t,u_t,x_t)+b(t,t,\lambda_t,u_t,x_t)\E[p(t)|\F_t]+\kappa(t,t,0,\lambda_t,u_t,x_t)\E[q(t,0)|\F_t]\lambda^B_t\nonumber\\
	&\quad +\int_{\R_0}\kappa(t,t,z,\lambda_t,u_t,x_t)\E[q(t,z)|\F_t]\lambda^H_t\nu(dz)\nonumber\\
	H_1^\bbF(t,\lambda,u,x,p,q)&:=\int_0^t\partial_t b(t,s,\lambda_s,u_s,x_s) ds\ \E[p(t)|\F_t] \\&\quad+ \int_0^t\!\!\!\int_\R \partial_t\kappa (t,s,z,\lambda_s, u_s,x_s) \E[\calD_{s,z}p(t)|\F_t]\Lambda(dsdz)
	\end{align*}
	
	\begin{notz}\label{notz}
	    Given $u,\hat u \in \calA^{\cdot}$, $X, \hat X$ represent the associated controlled dynamics of \eqref{dXt} and $(p,q), (\hat p,\hat q)$ are the corresponding solutions of \eqref{dYt}.
		From now on, if no confusion arises, we will use the compact notation: $$b(t,s):=b(t,s,\lambda_s,u(s),X(s)), \quad \hatb(t,s):=b(t,s,\lambda_s,\hatu(s),\hatX(s)).$$ 
  Similarly, for $\kappa$, $\hatkappa$, $F$, $\hat F$, $G$, $\hat G$, we will also write:
		$$\calH^{u}(s):=\calH(s,\lambda,u, X,\hat p,\hat q), \quad \calH^{\hat u}(s):=\calH(s,\lambda,\hat u, \hatX,\hat p,\hat q)$$
		and similarly for $\calH^{\bbF,u}$, $\calH^{\bbF,\hatu}$, $H_0^{u}$, $H_0^{\hat u}$, $H_0^{\bbF, u}$, $H_0^{\bbF, \hat u}$, $H_1^{u}$, $H_1^{\hat u}$, $H_1^{\bbF, u}$, $H_1^{\bbF, \hat u}$.
	\end{notz}
	
	\begin{teo}\label{MaximumF}\emph{\textbf{(Sufficient maximum principle with respect to $\bbF$).}}
 Let $\lambda \in \calL$.
		Let $\hat{u} \in\calA^\bbF$ and assume that the corresponding solutions $\hatX$, $(\hat p,\hat q)$ of \eqref{dXt} and \eqref{dYt} exist.
		Assume that 
		\begin{itemize}
			\item $x\longmapsto G(x)$ is concave. 
			\item For any $t$, the map
			\begin{equation}\label{ArrowConditionF}
			x\longmapsto\esssup_{u\in \Xi_{\calU}}\calH^\bbF(t,\lambda,u,x,\hat p,\hat q), \quad x \in \Xi_\R,
			\end{equation}
			is concave.
			\item For all $\t$,
			\begin{equation}\label{MaximumConditionF}
			\hspace{-.4 cm} 
   \esssup_{u\in \Xi_{\calU}}\calH^\bbF(t,\lambda,u,x,\hat p,\hat q) 
=\calH^\bbF(t,\lambda,\hatu,\hat X,\hat p,\hat q). 
			\end{equation}
		\end{itemize}
		Then $\hatu$ is an optimal control for problem \eqref{performancefunctionalG} and $(\hatu,\hatX)$ is an optimal pair.
	\end{teo}

\begin{proof} 
	This proof is inspired by both the proof of \cite{Oksendal1} Theorem 4.1 and \cite{DiNunno1} Theorem 6.2. The main difference with \cite{Oksendal1} is the use of the random measure $\mu$ instead of a Brownian motion and a compensated random Poisson measure, which requires to abandon the framework of Malliavin calculus. The main difference with \cite{DiNunno1} is the Volterra structure of the dynamics for the forward equation \eqref{dXt}, which lead to more involved stochastic calculus. Recall that $\hat u \in \calA^{\bbF}$ is a candidate to be optimal and $X^{\hat u}$ is the corresponding solution of \eqref{dXt}.
	Choose an arbitrary other $u\in\calA^\bbF$ with corresponding controlled dynamics $X$ and consider $J(u)-J(\hatu)=I_1+I_2$, where 
	\begin{align}
	I_1&:=\E\left[\int_0^T F(t,\lambda_t,u(t),X(t))-F(t,\lambda_t,\hatu(t),\hatX(t))dt\right], \label{I_1}\\
	I_2&:=\E\left[G(X(T))-G(\hatX(T))\right].
	\end{align}
	Considering now $I_1$, from the definition of $H_0^\bbF$ we get that,
	\begin{align*}
	I_1&=\E\left[\int_0^T\!\!\!\left\{H_0^{\bbF,u}(t)-H_0^{\bbF,\hatu}(t)-[b(t,t)-\hatb(t,t)]\E\left[\hat p(t)|\F_t\right]\right\}dt\right.\\
	&\left.\quad -\int_0^T\!\!\!\int_{\R}[\kappa(t,t,z)-\hatk(t,t,z)]\E\left[\hat  q (t,z)|\F_t\right]\Lambda(dtdz)\right].
	\end{align*}
	By the concavity of $G$, we have
	\begin{equation*}
	I_2\leq\E\left[\partial_xG(\hatX(T))\left(X(T)-\hatX(T)\right)\right]=\E\left[\hat p(T)\left(X(T)-\hatX(T)\right)\right].
	\end{equation*}
	We apply the transformation rule (Lemma \ref{Transformation rule}) to rewrite the Volterra forward dynamics of $X$ as 
	\begin{equation*}
	     dX(t)= \Big(b(t,t)  + \int_0^t\partial_t b(t,s) ds +\int_0^t\!\!\!\int_\R \partial_t \kappa (t,s,z)\mu(dsdz)\Big)dt+ \int_\R \kappa(t,t,z)\mu(dtdz).
	\end{equation*}
	Also the BSDE $\hat p$ \eqref{dYt} associated to the optimal pair $(\hatu, \hatX)$ in differential notation is
	\begin{align*}
	   d\hat p(t)&=-\partial_x\calH^{\hatu }(t)dt+\int_\R \hat q(t,z)\mu(dtdz).
	\end{align*}
	Using the Itô formula for the product we obtain
	\begin{align}
	I_2	&\leq\E\left[\int_0^T\!\!\left\{\hat p(t)\left(\left(b(t,t)-\hatb(t,t)\right)+\int_0^t\left(\partial_t b(t,s)-\partial_t \hatb(t,s)\right)ds\right.\right.\right.\nonumber\\
	&\left.\left.\quad +\int_0^t\!\!\!\int_{\R}\left(\partial_t \kappa(t,s,z)-\partial_t \hatk(t,s,z)\right)\mu(dsdz)\right)\right\}dt\nonumber\\
	&\quad -\int_0^T\partial_x \calH^{\hatu}(t)\left(X(t)-\hatX(t)\right) dt+\int_0^T\left\{[\kappa(t,t,0)-\hatk(t,t,0)]\hat q(t,0)\lambda^B_t\right.\nonumber\\
	&\left.\left.\quad+ \int_{\R_0}[\kappa(t,t,z)-\hatk(t,t,z)]\hat q(t,z)\nu(dz)\lambda^H_t\right\}dt\right]\label{I2_minoration}.
	\end{align}
	Now notice that, 
    \begin{align}
	\E\left[\int_0^T\left(\int_0^t\!\!\!\int_{\R}\partial_t \kappa(t,s,z)\mu(dsdz)\right)\hat p(t) dt\right]
	&=\int_0^T\E\left[\left(\int_0^t\!\!\!\int_{\R}\partial_t \kappa(t,s,z)\mu(dsdz)\right)\hat p(t)\right]dt\nonumber\\
	&=\int_0^T\E\left[\int_0^t\!\!\!\int_{\R}\partial_t \kappa(t,s,z)\calD_{s,z}\hat p(t)\Lambda(dsdz)\right]dt\nonumber\\
	&=\E\left[\int_0^T\!\!\!\int_0^t\!\!\!\int_\R\partial_t\kappa(t,s,z)\calD_{s,z}\hat p(t)\Lambda(dsdz)dt \right]\label{Ok40-41}
	\end{align}
	where we have used Fubini's theorem and the duality formula (Proposition \ref{Yab}). By substituting \eqref{Ok40-41} into \eqref{I2_minoration}, and taking the conditional expectation given $\F_t$ we get that
	\begin{align*}
	I_2&\leq \E\left[\int_0^T\left\{\left(b(t,t)-\hatb(t,t)\right)\E\left[\hat p(t)|\F_t\right]+\int_0^t\left(\partial_t b(t,s)-\partial_t \hatb(t,s)\right)ds\ \E\left[\hat p(t)|\F_t\right]\right.\right.\\
	&\left.\quad+\int_0^t\!\!\!\int_{\R_0}\left(\partial_t \kappa(t,s,z)-\partial_t \hatk(t,s,z)\right)\E\left[\calD_{s,z}\hat p(t)|\F_t\right]\Lambda(dsdz)\right\} dt\\
	&\quad\left.-\int_0^T\partial_x \calH^{\bbF, \hat u}(t)\left(X(t)-\hatX(t)\right)dt+\int_0^T\!\!\!\int_{\R}
	\E\left[\hat q(t,z)|\F_t\right][\kappa(t,t,z)-\hatk(t,t,z)]\Lambda(dtdz)\right].
	\end{align*}
	Hence
	\begin{align}
	I_1+I_2&\leq\E\left[\int_0^T\left(H_0^{\bbF,u}(t)-H_0^{\bbF,\hatu}(t)+H_1^{\bbF,u}(t)-H_1^{\bbF,\hatu}(t)-\partial_x \calH^{\bbF,\hatu}(t)\left(X(t)-\hatX(t)\right)\right)dt\right]\nonumber \\
	&=\E\left[\int_0^T\!\!\!\left(\calH^{\bbF,u}(t)-\calH^{\bbF,\hatu}(t)-\partial_x{\calH^{\bbF,\hatu}}(t)\left(X(t)-\hatX(t)\right)\right)dt\right]\leq 0,\label{IntegrandaMaximumPrinciple}
	\end{align}
	 $dt\times dP$ a.e. by the maximality of $\hatu$ in \eqref{MaximumConditionF} and the concavity condition \eqref{ArrowConditionF}. Hence $ J(u)\leq J(\hatu)$ and $\hatu$ is an optimal control for \eqref{performancefunctionalG}. This conclusion is reached applying a separating hyperplane argument to the concave map \eqref{ArrowConditionF}.	
\end{proof}

	Notice that a result analogous to Theorem \ref{MaximumF} can also be obtained when working under the initially enlarged filtration $\bbG$. Though the next result might not be of direct applicability in view of the anticipated information included in $\bbG$, the study has mathematical validity. 

 \begin{obs}
     The transformation rule under \eqref{lipschitz_partial_t} allows for the use of an Itô-type formula in the context of Volterra dynamics. If the equation would not present Volterra structure in the stochastic integral part (i.e. in the coefficient $\kappa$), then the requirement \eqref{lipschitz_partial_t} is clearly lifted.
 \end{obs}

	\begin{prop}\label{MaximumG}\emph{\textbf{(Sufficient maximum principle with respect to $\bbG$).}}
 Let $\lambda \in \calL$.
		Let $\hatu\in\calA^\bbG$ and assume that the corresponding solutions $\hatX(t), (\hat p,\hat q)$ of \eqref{dXt} and \eqref{dYt} exist.
		Assume that:
		\begin{itemize}
			\item $x\longmapsto G(x)$ is concave. 
			\item For any $t$, $\hat p$, $\hat q$, the function
			\begin{equation}\label{ArrowConditionG}
			x\longmapsto\esssup_{u\in\Xi_{\calU}}\calH(t,\lambda,u,x,\hat p,\hat q), \qquad x \in \xi_{\R}
			\end{equation}
			is concave in $x$.
			\item For all $\t$,
			\begin{equation}\label{MaximumConditionG}
			\esssup_{v\in\Xi_\calU}\calH(t,\lambda,v,\hatX,\hat p,\hat q)=\calH(t,\lambda,\hatu,\hatX,\hat p,\hat q). 
			\end{equation}
		\end{itemize}
		Then $\hatu$ is an optimal control for the problem \eqref{J(u) wrt G}.
	\end{prop}

	\begin{proof}
	Once considering the filtration $\bbG$, the arguments in the proof of Theorem \ref{MaximumF} leading to
	\begin{align*}
	J(u)&-J(\hatu)\leq \E\left[\int_0^T\left(\calH^u(t)-\calH^{\hatu}(t)-\partial_x \calH^{\hatu}(t)\left(X(t)-\hatX(t)\right)\right)dt\right]\leq 0
	\end{align*}
	apply directly without conditioning.
\end{proof}

	\section{Necessary maximum principles with time-change}
	Hereafter we study necessary conditions to identify the possible candidates for optimal controls. This can be a useful starting point before applying a verification theorem to ensure optimality. We remark that our results relax the condition of concavity present in Theorem \ref{MaximumF} and \ref{MaximumG}. However, we introduce some other assumptions on the set of admissible controls and the first variation process of the forward dynamics \eqref{dXt}. 
	
	In the literature we find a first version of necessary maximum principle for Volterra dynamics in \cite{Oksendal2}. There the driving noises were the Gaussian and the centered Poisson random measure. Our work goes beyond these noises.\\
	\noindent For any $\t$, we consider a random perturbation of the type
	\begin{equation}\label{beta}
	\beta(s):=\alpha_t\mathds{1}_{[t,t+h]}(s), \quad s\in[0,T],
	\end{equation}
	where $\alpha_t$ is a bounded $\F_t$ measurable random variable and $h\in[0, T-t]$. We make the following assumptions:
	\begin{enumerate}
		\item The set of admissible controls $\calA^{\bbF}$ is such that, for all $u\in\calA^\bbF$,
		\begin{equation*}
		u+\varepsilon\beta\in\calA^\bbF,
		\end{equation*}
		for all perturbations $\beta$ as in \eqref{beta} and all $\varepsilon>0$ sufficiently small.
		\item	The first variation process $\rchi(t)$, $\t,$ given by the derivative
		\begin{equation}\label{First Variation}
		\rchi(t):=\partial_\varepsilon X^{(u+\varepsilon\beta)}|_{\varepsilon=0}
		\end{equation}
		(see \eqref{dXt}) exists and is well defined.
		\item $\partial_x b(t,s)$ and $\partial_u b(t,s)$ are well defined and $C^1$ with respect to $t$ with partial derivatives $L^2$-integrable with respect to $dt\times dP$. $\partial_x \kappa(t,s,\cdot)$ and $\partial_u\kappa(t,s,\cdot)$ are well defined and $C^1$ with respect to $t$ with partial derivatives $L^2$-integrable with respect to $d\Lambda\times dP$.
		\item $\partial_x\kappa(t,s,z)$ and $\partial_u\kappa(t,s,z)$ are such that, for all $z\in\R$ $\lambda \in [0,\infty)^2$, $u\in\calU$, $x\in\R$, the partial derivative of $\partial_x\kappa+\partial_u\kappa$ with respect to $t$ is locally bounded (uniformly in $t$) and satisfies
	\begin{equation*}
	    |\partial_t(\partial_x\kappa(t_1,s,z,)+\partial_u\kappa(t_1,s,z))- \partial_t(\partial_x\kappa(t_2,s,z,)+\partial_u\kappa(t_2,s,z))|\leq K |t_1-t_2|,
	\end{equation*}
	for some $K>0$ and for each fixed $s\leq t$, $\lambda \in [0,\infty)^2$, $u\in\calU$, $x\in\R$.
	\end{enumerate}
	
\noindent Assumption 2. above implies that
	\begin{align*}
	\rchi(t)&=\int_0^t\Big(\partial_x b(t,s)\rchi(s)+\partial_u b(t,s)\beta(s) \Big)ds\\
	&+\int_0^t\!\!\!\int_{\R}\Big(\partial_x \kappa(t,s,z)\rchi(s)+\partial_u \kappa(t,s,z)\beta(s)\Big)\mu(dsdz),
	\end{align*}
	exists and is well defined, whereas assumption 4. ensure us to be able to apply the transformation rule for $\rchi$.

	Remark that sufficient conditions that ensure the existence of the first variation process are that  $b$ and $\kappa$ are in $C^1(\calU)$ uniformly for all $s,\t $ $\lambda \in [0,\infty)^2$, $x\in\R$ and that $(\partial_x b(t,s)\rchi(s)+\partial_u b(t,s)\beta(s))$ and $(\partial_x \kappa(t,s,z)\rchi(s)+\partial_u \kappa(t,s,z)\beta(s))$ satisfy the linear growth and lipschitzianity conditions of Theorem \ref{Esistenza Forward}. 
	
	As above we consider the performance functional \eqref{performancefunctionalG} with the related conditions on $F$ and $G$ as in Section \ref{SecSuff}. We also continue using the compact notation there introduced, see Notation \ref{notz}.	
 
	\begin{teo}\label{NecessaryMaximumTheorem_F}\emph{\textbf{(Necessary maximum principle with respect to $\bbF$).}}
 Let $\lambda \in \calL$.
		Suppose that $\hatu\in\calA^\bbF$ and the corresponding solutions $\hatX,(\hat p,\hat q)$ of \eqref{dXt} and \eqref{dYt} exist. Assume also that $P$-a.s. $F\in C^1(\calU)$ for all $\t $ $\lambda \in [0,\infty)^2$, $x\in\R$. If, for all perturbations $\beta$ as in \eqref{beta}, we have that
		\begin{equation}\label{condizionederivataF}
		\partial_\varepsilon J(\hatu+\varepsilon\beta)|_{\varepsilon=0}=0,
		\end{equation}
		then 
		\begin{equation}\label{NecessaryMaximumPrincipleF}
		\partial_u\calH^{\bbF,\hat u}(t)=0.
		\end{equation}
		The converse also holds true.
	\end{teo}

	\begin{proof}		
		With \eqref{First Variation}, we consider for $u\in \calA^\bbF$ and the perturbation \eqref{beta},
		\begin{align}\label{Variation J(u)}
	&	\partial_\epsilon J(u+\varepsilon\beta)|_{\varepsilon=0}  \\
  &=\E\left[\int_0^T\left(\partial_x F(t, \lambda_t, u(t), X(T))\rchi(t)+\partial_u F(t, \lambda_t, u(t), X(t))\beta(t)\right)dt+\partial_xG(X(T))\rchi(T)\right]. \notag
		\end{align}
		By considering a suitable increasing family of stopping times converging to $T$ as in  \cite{Stopping_times} Theorem 2.2, we may assume that all the local martingales appearing here are true martingales. 
		From \eqref{hamiltonF}, the transformation rule (Lemma \ref{Transformation rule}) and the Itô formula for the product, we find that 
		\begin{align*}
		\E&\left[\partial_x G(X(T))\rchi(T)\right]=\E\left[p(T)\rchi(T)\right]\\
		&=\E\left[\int_0^Tp(t)\left(\partial_x b(t,t)\rchi(t)+\partial_u b(t,t)\beta(t)\right)dt -\int_0^T\rchi(t)\partial_x{\calH}(t)dt \right.\\
		&\quad+ \int_0^Tp(t)\left(\int_0^t\left(\partial_t \partial_xb(t,s)\rchi(s)+\partial_t \partial_ub(t,s)\beta(s)\right)ds\right)dt\\
		&\quad+ \int_0^Tp(t)\left(\int_0^t\!\!\!\int_{\R}\left(\partial_t \partial_x\kappa(t,s,z)\rchi(s)+\partial_t \partial_u\kappa(t,s,z)\beta(s)\right)\mu(dsdz)\right)dt \\
		&\left.\quad+ \int_0^T\!\!\!\int_{\R}q(s,z)\Big(\partial_x \kappa(t,t,z)\rchi(t)+\partial_u \kappa(t,t,z)\beta(t)\Big)\Lambda(dtdz) \right].
		\end{align*}
		Now, recalling equality \eqref{Ok40-41}, and taking the conditional expectation under $\F_t$  we get that
		\begin{align*}
		\E&\left[p(T)\rchi(T)\right]\\
		&=\E\left[\int_0^T\left\{\partial_x b(t,t)\E\left[p(t)|\F_t\right]+\int_0^t\partial_x\partial_tb(t,s)ds\ \E\left[p(t)|\F_t\right] \right.\right.\\
		&\quad \left.+\int_0^t\!\!\!\int_\R \partial_x\partial_t\kappa(t,s,z)\E\left[\calD_{s,z}p(t)|\F_t\right]\Lambda(dsdz)\right\}\rchi(t)dt\\
		&\quad +\int_0^T\left\{\partial_u b(t,t)\E\left[p(t)|\F_t\right]+\int_0^t \partial_u\partial_tb(t,s)ds\ \E\left[p(t)|\F_t\right] \right.\\
		&\quad \left.+\int_0^t\!\!\!\int_{\R}\partial_u\partial_t \kappa(t,s,z)\E\left[\calD_{s,z}p(t)|\F_t\right]\Lambda(dsdz) \right\}\beta(t)dt-\int_0^T\partial_x{\calH}^{\bbF,u}(t)\rchi(t)dt\\
		&\quad \left.+\int_0^T\!\!\!\int_{\R}\left(\partial_x\kappa(t,t,z)\rchi(t)+\partial_u\kappa(t,t,z)\beta(t)\right)\E\left[q(t,z)|\F_t\right]\Lambda(dtdz)\right].
		\end{align*}
		So that, from \eqref{hamiltonF}, we can write
		\begin{align}
		\E&\left[\int_0^T\left(\partial_x F(t, \lambda_t, u(t), X(t))\rchi(t)+\partial_u F(t, \lambda_t, u(t), X(t) )\beta (t)\right)dt+\partial_x G(X(T))\rchi(T)\right]\nonumber \\
		&=\E\left[\int_0^T\partial_x \calH^{\bbF, u }(t)\rchi(t)dt-\int_0^T\partial_x \calH^{\bbF, u }(t)\rchi(t)dt+\int_0^T\partial_u\calH^{\bbF,u }(t)\beta(t)dt\right]\label{NEC rewrite}.
		\end{align}
		Summarizing, equation \eqref{Variation J(u)} together with \eqref{NEC rewrite} and the perturbations in \eqref{beta} give
		\begin{equation}\label{Ok53}
	\partial_\varepsilon J(u+\varepsilon\beta)|_{\varepsilon=0}=\E\left[\int_0^T\partial_u\calH^{\bbF, u }(t)\beta(t)dt\right]=\E\left[\int_t^{t+h}\partial_u\calH^{\bbF, u }
		(s)ds\ \alpha_t\right],
		\end{equation}
		and, for $\hatu$, \eqref{condizionederivataF} gives
		\begin{equation*}
		\partial_\varepsilon J(\hatu+\varepsilon\beta)|_{\varepsilon=0}=0.
		\end{equation*}
		Applying the Fubini theorem to the right-hand side of \eqref{Ok53} and differentiating at $h=0$ we obtain
		\begin{equation*}\label{Ok54}
		\E\left[\partial_u\calH^{\bbF, \hat u }(t) \ \alpha_t\right]=0,
		\end{equation*}
		for all $\alpha_t$ bounded and $\F_t$ measurable. Hence
		\begin{equation}\label{Ok55}
		\E\left[\partial_u\calH^{\bbF, \hat u }(t)\middle|\F_t\right] = \partial_u\calH^{\bbF, \hatu}(t)=0.
		\end{equation}
		Vice versa, if \eqref{Ok55} holds, we can reverse the argument to obtain \eqref{condizionederivataF}.
	\end{proof}
	
	As in Section 3, for the sake of completeness, we propose a necessary maximum principle under the information flow $\bbG$. This refers to the optimization problem \eqref{J(u) wrt G}. In this case we assume that, for all $u\in\calA^{\bbG}$, 
	$u+\varepsilon \beta \in\calA^{\bbG}$
	for all perturbations $\beta$ as in \eqref{beta} and $\varepsilon>0$ sufficiently small.
	
	\begin{prop}\label{NecessaryMaximumTheorem_G} \emph{\textbf{(Necessary maximum principle with respect to $\bbG$).}}	Let $\lambda \in \calL$.
		Suppose that $\hatu\in\calA^\bbG$ and the corresponding solutions $\hatX,(\hat p,\hat q)$ of \eqref{dXt} and \eqref{dYt} exist. Also assume that $F\in C^1(\calU)$ for all $\t $ $\lambda \in [0,\infty)^2$, $x\in\R$. If, for all perturbations $\beta$,
		\begin{equation}\label{condizionederivata_G}
		\partial_\varepsilon J(\hatu+\varepsilon\beta)|_{\varepsilon=0}=0,
		\end{equation}
		then 
		\begin{equation}\label{NecessaryMaximumPrinciple_G}
		\partial_u\calH^{\hat u}(t)=0.
		\end{equation}
		Conversely, if \eqref{NecessaryMaximumPrinciple_G} holds, then \eqref{condizionederivata_G} is true.		
	\end{prop}	
	\begin{proof}
		The argument in the proof of Theorem \ref{NecessaryMaximumTheorem_G} leading to 
		\begin{equation*}
		\partial_\epsilon J(\hatu+\epsilon \beta)|_{\varepsilon=0}=\E\left[\int_0^T\partial_u \calH^{\hat u}(t)\beta(t)dt\right]
		\end{equation*}
		still holds with no need to use conditional expectations. Since $\hatu$ and $\hatu+\epsilon \beta $ are $\bbG$-predictable, we obtain
		\begin{align*}
		\E&\left[\int_0^T\left\{\partial_u \hatb(t,t)\hat p (t)+\int_0^t\partial_u\partial_t \hatb(t,s) ds\hat p (t)+\int_0^t\!\!\!\int_{\R}\partial_u\partial_t \hatkappa(t,s,z)\calD_{s,z}\hat p (t) \Lambda(dsdz) \right\}\beta(t)dt\right.\\
		&\left. +\int_0^T\!\!\!\int_{\R}\left(\partial_u \hatkappa(t,t,z)\beta(t)\right)\hat q(t,z)\Lambda(dtdz)\right]=\E\left[\int_0^T\partial_u \calH^{\hat u }(t)\beta(t)dt\right],
		\end{align*}
		where we have used the definition of $\calH$ as in \eqref{hamiltonG}. We conclude as in Theorem \ref{NecessaryMaximumTheorem_F}.
	\end{proof}

\section{A maximum principle approach in optimal harvesting }
\label{FISH}
\label{sec: example}
We now go back to the optimal harvesting problem within fishery, where the population dynamics is given  by the dynamics \eqref{X(t) example}. 
We recall that our starting point are \cite{VolterraFish,VolterraFish2,VolterrraIntegral}, where the authors consider deterministic Volterra models to model population growth and, following e.g. \cite{StochasticFish,StochasticFish2}, we introduce some random fluctuations that will affect the population growth. Hence, the dynamics considered are of type \eqref{X(t) example}:
	\begin{equation}\label{Xu_last_example}
	    X^u(t) = X_0+\int_0^t \left(r(t,s) - K u(s)\right)X^u(s) ds +\int_0^t\sigma(s)X^u(s)dB(s), \quad \t,
	\end{equation}
	where $r(t,s):[0,T]^2\longrightarrow\R$, $\sigma(s):[0,T]\longrightarrow\R$, $K>0$, $X_0>0$. Here, $B$ is the conditional Gaussian measure.
We assume that \eqref{Xu_last_example} admits a solution, that $r(t,s)$ is $C^2$ with respect to both $t$ and $s$, and that $\sigma$ is $C^1$ with respect to $t$ and $\sigma(t) > -1 $ for all $s\in[0,T]$, $z\in\R$. Lastly we assume $r(t,s)$, $\partial_t r(t,s)$ and $\sigma(t)$ are in $L^2(dt)$. For sufficient conditions that guarantee the existence of a solution of $X$ we refer to Theorem \ref{Esistenza Forward}.
In the context of optimal harvesting of fish, $r$ represents the growth rate, $K$ the catchability coefficient, and the control $u$ is the fishing effort.
	
\vspace{2mm}
Let us define
	\begin{equation*}
	    \tau := \inf\{\t, \text{ such that } X^u(t) = 0\} \wedge T.
	\end{equation*}
	Then we can see that $X^u(t) = 0$ for all $t\geq \tau$. In fact , for $0\leq \tau \leq t \leq T$, we have that \eqref{X(t) example} can be rewritten as 
	\begin{align*}
	    X(t) &= X_0 + \int_0^\tau (r(\tau,s) - K u(s))X(s) ds + \int_0^\tau \sigma(s)X(s) dB(s) \\&\quad + \int_\tau^t (r(t,s)-Ku(s))X(s) ds + \int_\tau ^t\sigma(s)X(s) dB(s) + \int_0^\tau (r(t,s)-r(\tau,s))X(s) ds\\
	    &= X(\tau) +\int_\tau^t (r(t,s)-Ku(s))X(s) ds + \int_\tau ^t\sigma(s)X(s) dB(s) + \int_0^\tau (r(t,s)-r(\tau,s))X(s) ds,
	\end{align*}
Being $X(0)=X_0 >0$ and the process $X$ continuous,  we have that $X$ is strictly positive, up to restricting ourselves to the interval $[0,\tau]$.

	Our goal is to characterise the optimal solution to maximization of the performance functional
	\begin{equation}\label{J(u) example}
	    J(u) = \E\left[\int_0^Te^{-\delta(T-t)}X(t)u(t)dt \right],
	\end{equation}
	where $u\in\calA^{\bbF}$, $\delta  >0 $.
 In the context of oprimal harvesting this can be regarded as the aggregated net discounted revenue, see \cite{perffunc}.
Following the approach given in this work, we consider the Hamiltonian functional \eqref{hamiltonG}, which can be here rewritten as 
	\begin{align*}
	    \calH^u(t)  &= e^{-\delta(T-t)}u(t)X(t) + \big(r(t,t)-Ku(t)\big)X(t)  p (t) + \sigma(t)X(t)q(t)\lambda_t^B \\ & \quad 
	    + \int_0^t\partial_t r(t,s)X(s)-Ku(s) ds \  p (t),
	\end{align*}
	where the backward dynamics for $ p $ are given by
	\begin{align}\label{Y(t) example}
	    d p (t)&= e^{-\delta(T-t)}u(t)dt+\Bigg(r(t,t) + \int_0^t \partial_t r(t,s) ds \Bigg)  p (t) dt + \sigma(t)q(t)\lambda_t^B dt+  q(t)dB(t)\nonumber \\
	     p (T) &= 0.
	\end{align}
Also, we consider the mapping $\calH^\bbF$ in \eqref{hamiltonF}:
	\begin{align*}
	    \calH^{\bbF, u}(t)  &= e^{-\delta(T-t)}u(t)X(t) + \big(r(t,t)-K u(t)\big)X(t) \E[ p (t)|\F_t] \\ & \quad + \sigma(t)X(t)\E[q(t)|\F_t]\lambda_t^B 
	    + \int_0^t\partial_t r(t,s)X(s)- K u(s) ds \ \E[ p (t)|\F_t].
	\end{align*}
	From Theorem \ref{NecessaryMaximumTheorem_F} we see that a necessary condition for an admissible control  $\hatu$ to be optimal is that, for all $\t$, $\partial_u\calH^{\bbF,\hat u}(t)=0$.
	Furthermore, from Theorem \ref{MaximumF}, being the map \eqref{ArrowConditionF} trivially concave, the condition $\partial_u \calH^{\bbF,\hat u}(t)=0$ is also sufficient for the maximality. In particular this means that an admissible control $\hatu$ is optimal if and only if
	\begin{equation}\label{NSC}
	    e^{-\delta(T-t)}\hatX(t) = K \hatX(t)\E[\hat  p (t)|\F_t].
	\end{equation}
Namely, for all $t\in [0,\tau]$
    \begin{equation}\label{NSC_rewritten}
 \E[\hat  p (t)|\F_t] =  K^{-1 }e^{-\delta(T-t)} .
 \end{equation}
 
	To find a solution to \eqref{Y(t) example} with respect to the information flow $\bbG$, we use a Girsanov change of measure as presented in \cite{DiNunnoGirsanov}. Define the measure $\mathbb{Q}$ by 
$ 	    d\mathbb{Q} = \mathcal{M}(T) dP(T)\quad \text{on} \ \calG_T,$
	where 
	\begin{align}
	    d\mathcal{M}(t) &= \mathcal{M}(t)\sigma(t)dB(t)\nonumber \\
	    \mathcal{M}(0)& = 1.\label{girsanov M}
	\end{align}
An explicit solution for \eqref{girsanov M} is obtained by the It\^o formula (see \cite{DiNunnoGirsanov}) and is given by
	\begin{align*}
	    \mathcal{M}(t) &= \exp\Bigg\{\int_0^t\sigma(s) dB(s) - \int_0^t\frac 1 2 \sigma(s)^2\lambda_s^B ds\Bigg\} , \quad t\in [0,T].
	\end{align*}
	We thus have that, under the measure $\mathbb{Q}$,
	\begin{equation*}
	    dB^\sigma(t) = dB(t) - \sigma(t)\lambda^B_t dt,
	\end{equation*}
	is a $\bbG$-martingale. Equation \eqref{Y(t) example} can now be rewritten under $\mathbb{Q}$ as
	\begin{align}\label{Y(t) example Q}
	    d \hat p (t)&= e^{-\delta(T-t)} \hatu(t)dt + \Bigg(r(t,t) + \int_0^t \partial_t r(t,s) ds \Bigg)  \hat p (t) dt + \hat q (t)dB^\sigma(t)\nonumber \\
	     \hat p (T) &= 0,
	\end{align}
	Thanks to \cite{DiNunno1} we know that \eqref{Y(t) example Q} admits a unique solution $(\hat p, \hat q)$ and that the process $ \hat p $ is given by 
    \begin{align}
     \hat p (t) &= \E_\mathbb{Q}\Bigg[\int_t^T\exp\left\{\int_t^s\tilde{r}(v) dv \right\}e^{-\delta(T-s)}\hatu(s)ds\Bigg],\nonumber
    \end{align}
    where we defined $\tilde r(t) := r(t,t) + \int_0^t \partial_t r(t,s) ds$. We thus obtain that
    \begin{align}
    \E\Bigg[ \hat p (t)|\F_t\Bigg] &= \E\Bigg[\frac{1}{\mathcal{M}(T)}\int_t^T\exp\left\{\int_t^s\tilde{r}(v) dv \right\}e^{-\delta(T-s)} \hatu(s)ds\Big|\F_t\Bigg]. \label{soluzione_esempio}
    \end{align}
       Substituting \eqref{soluzione_esempio} in \eqref{NSC_rewritten} we obtain a characterization of $\hatu(t)$.
	
	\section*{Declarations}
	\noindent\textbf{Funding.} The research leading to these results received funding from the Research Council of Norway within the project STORM: Stochastics for Time-Space Risk Models, grant number: 274410.

    \noindent\textbf{Conflicts of interest.}
    The authors have no competing interests to declare that are relevant to the content of this article.

	\bibliographystyle{plain}
	\bibliography{bib}

\end{document}